\documentclass[a4, 10pt,leqno]{amsart}
\usepackage{pifont}
\usepackage{mathtools}
\usepackage{mathabx}
\usepackage{bbm}
\usepackage{amsthm}
\usepackage{mathtools}
\usepackage{mathtools,accents}
\usepackage{amssymb,amscd,amsthm,enumerate}
\usepackage{mathrsfs}
\usepackage{a4wide}
\usepackage{eqnarray}
\usepackage{enumitem}
\usepackage[utf8]{inputenc}
\usepackage[T1]{fontenc}
\usepackage{lmodern}
\usepackage[english]{babel}
\usepackage{microtype}
\usepackage{float}
\usepackage{esint}
\usepackage{physics}
\usepackage{a4wide}
\usepackage{algorithmic}
\usepackage[ruled]{algorithm2e}
\usepackage[active]{srcltx}
\usepackage{lipsum}
\usepackage{upgreek}
\usepackage{listings}
\usepackage{wasysym}
\usepackage{csquotes}
\usepackage{enumerate}
\usepackage{color}
\usepackage[usenames,dvipsnames]{xcolor}
\usepackage{xcolor}
\definecolor{ultrablue}{rgb}{0.0,0.0, 1}
\definecolor{jigari}{rgb}{0.39,0.0, 0.0}
\usepackage{url}
\usepackage{hyperref}
\hypersetup{
	linktoc=page,
	colorlinks,
	linkcolor={ultrablue},
	citecolor={ultrablue},
	urlcolor={ultrablue}
}
\hypersetup{breaklinks=true}
\usepackage{soul}
\usepackage{graphicx}
\usepackage{tikz}
\usepackage{tikzpagenodes}
\usepackage{scrlayer-scrpage}
\usepackage{tikz}
\usetikzlibrary{calc,spy}
\usetikzlibrary{shapes.geometric, arrows}
\usepackage{pgfplots}
\usetikzlibrary{intersections, pgfplots.fillbetween}
\usepackage{tikz}
\usetikzlibrary{matrix}
\usepackage[all]{xy}
\usepackage{subfig}
\newlength{\myeqskip}  
\AtBeginDocument{%
	\setlength\abovedisplayskip{\myeqskip}%
	\setlength\belowdisplayskip{\myeqskip}%
	\setlength\abovedisplayshortskip{\myeqskip-\baselineskip}%
	\setlength\belowdisplayshortskip{\myeqskip}}
\allowdisplaybreaks
\numberwithin{equation}{section}
\usepackage[a4paper,bindingoffset=0in,%
left=1.2in,
right=1.2in,
top=1.6in,
bottom=1.6 in,%
footskip=1mm]{geometry}
\theoremstyle{plain}
\newtheorem*{theorem*}{Main Theorem}
\newtheorem{lemma}{Lemma}[section]
\newtheorem{theorem}[lemma]{Theorem}

\theoremstyle{definition}
\newtheorem{definition}[lemma]{Definition}
\newtheorem{remark}[lemma]{Remark}

\theoremstyle{remark}

\newcounter{example}

\makeatletter
\def\@xnamedef#1{\expandafter\protected@xdef\csname #1\endcsname}
\def\no@harm{} 
\def\ead@au#1{\protected@edef\@ead@au{#1}}
\patchcmd\runningauthor@fmt{\global\edef}{\protected@xdef}{}{}
\patchcmd\runningauthor@fmt{\global\edef}{\protected@xdef}{}{}
\patchcmd\author@fmt{\edef}{\protected@edef}{}{}
\patchcmd\add@xtok{\xdef}{\protected@xdef}{}{}
\makeatother
\newcommand{\secref}[1]{\textsection\thinspace\ref{#1}}

\newcommand{\R}{\mathbb R}
\newcommand{\Sp}{\mathbb S}

\DeclareMathOperator{\proj}{proj}
\DeclareMathOperator{\Hess}{Hess}

\DeclareMathOperator{\Ric}{Ric}
\DeclareMathOperator{\Riem}{Riem}
\DeclareMathOperator{\scal}{scal}

\newcommand{\K}{\mathcal K}
\newcommand{\y}{{\footnotesize \textbf{\textit{y}}}}
\newcommand{\vecv}{\footnotesize \textbf{\textit{v}}}
\newcommand{\vecw}{{\footnotesize \textbf{\textit{w}}}}

\newcommand{\dist}{\mathrm{d}}

\newcommand{\cpt}{{\textsf{CPT}}\,}
\usepackage{accents}
\newcommand*{\dt}[1]{%
		\accentset{\mbox{\bfseries .}}{#1}}
\newcommand*{\ddt}[1]{%
	\accentset{\mbox{\bfseries .\hspace{-0.25ex}.}}{#1}}
\newcommand*{\dddt}[1]{%
	\accentset{\mbox{\bfseries .\hspace{-0.25ex}.\hspace{-0.25ex}.}}{#1}}

\makeatletter
\newcommand*\bigcdot{\mathpalette\bigcdot@{.5}}
\newcommand*\bigcdot@[2]{\mathbin{\vcenter{\hbox{\scalebox{#2}{$\m@th#1\bullet$}}}}}
\makeatother
\newcount\stylenum  \newdimen\styledim 
\def\varstyle#1{\mathchoice{\stylenum=0 #1}{\stylenum=1 #1}{\stylenum=2 #1}{\stylenum=3 #1}}
\def\mathaxis{\fontdimen22\ifcase\stylenum 
	\textfont\or\textfont\or\scriptfont\or\scriptscriptfont\fi2 }
\def\setstyledim{\styledim=\ifcase\stylenum .1em\or.1em\or.07em\or.05em\fi\relax}
\def\sqdot{\mathbin{\varstyle{\raise\mathaxis\hbox{\setstyledim
				\kern\styledim 
				\vrule width1.2\styledim height.6\styledim depth.6\styledim
				\kern\styledim}}}}

		\DeclareMathAlphabet{\mathdutchcal}{U}{dutchcal}{m}{n}
		\SetMathAlphabet{\mathdutchcal}{bold}{U}{dutchcal}{b}{n}
		\DeclareMathAlphabet{\mathdutchbcal}{U}{dutchcal}{b}{n}
		\DeclareSymbolFont{myletters}{OML}{ztmcm}{m}{it}
		\DeclareMathSymbol{\nicelambda}{\mathord}{myletters}{"15}
		\usepackage{nicefrac}
\usepackage{stackengine}
\setstackEOL{\\}
\newcounter{tmpctr}
\newcommand\fancyRoman[1]{%
	\setcounter{tmpctr}{#1}%
	\setbox0=\hbox{\kern.2pt\textsf{\Roman{tmpctr}}}%
	\setstackgap{S}{-.6pt}%
	\Shortstack{\rule{\dimexpr\wd0+.1ex}{.7pt}\\\copy0\\
		\rule{\dimexpr\wd0+.1ex}{.7pt}}%
}
\newcommand{\bbGamma}{{\mathpalette\makebbGamma\relax}}
\newcommand{\makebbGamma}[2]{%
	\raisebox{\depth}{\scalebox{1}[-1]{$\mathsurround=0pt#1\mathbb{L}$}}%
}
\newcommand{\eps}{\varepsilon}
		\usepackage{pifont}
		
\newcommand{\restr}{\raisebox{-.1908ex}{$\big|$}}
\makeatother
\newcommand{\ident}{\raisebox{0pt}{\scalebox{1.1}{$\mathbbm{1}$}}\hspace{-1pt}}

\makeatletter
\newcommand{\tpitchfork}{%
	\vbox{
		\baselineskip\z@skip
		\lineskip-.52ex
		\lineskiplimit\maxdimen
		\m@th
		\ialign{##\crcr\hidewidth\smash{$-$}\hidewidth\crcr$\pitchfork$\crcr}
	}%
}
\makeatletter
\newcommand{\tpmod}[1]{{\@displayfalse\pmod{#1}}}
\makeatother

\usepackage{stackengine,scalerel}

\newcommand\overstar[1]{\ThisStyle{\ensurestackMath{%
			\setbox0=\hbox{$\SavedStyle#1$}%
			\stackengine{0pt}{\copy0}{\kern0\ht0\smash{\SavedStyle\star}}{O}{c}{F}{T}{S}}}}

\newcommand\dunderline[3][-1pt]{{%
		\sbox0{#3}%
		\ooalign{\copy0\cr\rule[\dimexpr#1-#2\relax]{\wd0}{#2}}}}
\makeatletter
\@namedef{subjclassname@2020}{%
	\textup{2020} Mathematics Subject Classification}
\makeatother
\makeatletter 
\def\l@subsection{\@tocline{2}{0pt}{3pc}{6pc}{}}
\makeatother
\makeatletter 
\def\l@subsection{\@tocline{2}{0pt}{3pc}{6pc}{}}
\makeatother

\def\bysame{\leavevmode\hbox to3em{\hrulefill}\thinspace}

\begin{document}
\title[\scriptsize {On classification of Finslerian spaces with nontrivial concircular transformations}]{\small On classification of Finslerian spaces \\ with nontrivial concircular transformations} 
%

\author[\protect \scriptsize Z. Fathi]{Zohreh Fathi}
\author[\protect \scriptsize S. Lakzian]{Sajjad Lakzian$^{^{\scalebox{1}{$\star$}}}$}
\address{\noindent -- Z. Fathi \& S. Lakzian \newline \noindent School of Mathematics, \newline Institute for Research in Fundamental Sciences (IPM), \newline P. O. Box 19395-
	5746, Tehran, Iran}
\address{\noindent -- S. Lakzian \newline \noindent Department of Mathematical Sciences\newline Isfahan University of  Technology (IUT) \newline Isfahan 8415683111, Iran.}
\email{\href{mailto:slakzian@iut.ac.ir}{slakzian@iut.ac.ir}}
\thanks{ZF is partially supported by IPM grant No. 1403530043 and SL is partially supported by IPM grant No. 1405530313.}
\thanks{SL is supported by the INSF Grant No. 4030556, awarded by the “On
	the Frontiers of Mathematical Sciences” program.}
\thanks{$\raisebox{2pt}{$\star$}$\textit{the corresponding author}}
\maketitle
%
\begin{abstract}
\par \textsl{We prove that in a non-trivial conformal circle-preserving transformation (\cpt for short), $e^{\sigma}F$ on a (forward or backward) complete Finslerian manifold $(M,F)$, the conformal factor has at most two critical points. Then a diffeomorphism classification based on the number of critical points as well as some curvature rigidity properties - of Finslerian manifolds that admit nontrivial conformal \cpt\!s - is presented.}
\end{abstract}
\date{\today}
\section{Introduction}\label{sec:intro}
As an immediate generalization of a geodesic in a Riemannian manifold $(M^n,g)$, one can consider geodesic circles that are smooth curves with constant first geodesic curvature $\kappa \ge 0$, and vanishing second geodesic curvature. Therefore, a unit-speed geodesic circle $c(s)$ solves the ODE
\begin{align*}
c'''(s) + \kappa^2 c'(s) = 0, \quad \text{where} \quad ' := \nabla_{c'}.
\end{align*}
\par Circle-preserving transformations ({\cpt}\!s) as diffeomorphisms that preserve geodesic circles (do not necessarily preserve the parametrization) are the main objects of study in concircular geometry. Concircular geometry was pioneered by Yano~\cite{Yano1}--\cite{Yano5}; important contributions to the theory include~\cite{Tachib,Nag,Ish-Tash,Ish,Tash1,Tash2,VO,NM,Kan,Fer,Kuh}.
\par Let $\varphi$ be a $\cpt$\!, then it is by now standard that $\varphi$ is a conformal transformation~\cite{VO}; thus, $\varphi^*(g) = \rho^{-2}g$ for some smooth function $\rho$ 
that solves the PDE 
\begin{align*}
\nabla^2 \rho = \nicefrac{1}{n} \left(\Delta \rho\right) g, \quad \rho>0,
\end{align*}
or equivalently
\begin{align}\label{eq:Riem-pde}
	\nabla^2 \sigma - d\sigma^2 = \nicefrac{1}{n}\left(\Delta \sigma - \|\nabla \sigma\|^2 \right) g, \quad \sigma := - \ln \rho;
\end{align}
see~\cite{Yano2, Kuh}.
Incidentally, this is also the PDE that describes Einstein conformal transformations of an Einstein metric $g$; see~\cite{Br}.
\par As a result of the collected works~\cite{Tash1,Tash2,Kan, Kuh}, there is a conformal diffeomorphic classification available in the complete setting. Indeed, letting ${\sf Crit}$ denote the set of critical points of $\sigma$, we have $\texttt{\#} ({\sf Crit}) \le 2$, and based on the cardinality of ${\sf Crit}$, one of the following three cases can occur.
\begin{itemize}
	\item[\large $\bullet$] Case \textsf{I}: ${\sf Crit} = \varnothing$. 
	\par In this case, $(M^n,g)$ is complete, non-compact, and conformally diffeomorphic to an isometric product $\R \times \Sigma^{n-1}$ in which $\Sigma^{n-1}$ is complete. 
	\medskip
	\item[\large $\bullet$] Case \textsf{II}: $\texttt{\#} \, {\sf Crit} = 1$. 
	\par In this case, $(M^n,g)$ is complete, non-compact, and conformally diffeomorphic to either the flat Euclidean space or the hyperbolic space.
	\medskip
	\item[\large $\bullet$] Case \textsf{III}: $\texttt{\#} \, {\sf Crit} = 2$. 
	\par In this case, $(M^n,g)$ is conformally diffeomorphic to the standard sphere; furthermore, if $g$ is of constant scalar curvature, then $(M^n,g)$ is homothetic to the standard sphere. 
\end{itemize}
In all cases, $M\smallsetminus {\sf Crit}$ is diffeomorphic to $I \times \Sigma$ and the metric has a warped product decomposition; where $I$ is an open interval and $\Sigma$ is the ideal regular level hypersurface of $\sigma$ (such level hypersurfaces are all diffeomorphic). In addition, regarding the curvature rigidity, the following hold true:  
\smallskip
\par \emph{The metric $g$ is Einstein (resp. of constant sectional curvature) if and only if the cross sections $\Sigma$ are Einstein (resp. of constant sectional curvature) and the normalized scalar curvature equals $-\nicefrac{\rho_{uuu}}{\rho_u}$, where $u$ is the coordinate that integrates $\partial_u = \nicefrac{\grad\sigma}{\|\grad\sigma\|}$}.
\par \emph{Moreover, $(M,g)$ is isometric to a standard Euclidean or hyperbolic or sphere or an isometric product $I \times \Sigma$ where $\Sigma$ is Ricci-flat and $M$ has negative scalar (or Ricci) curvature.} 
\subsection*{{\cpt}\!s in the Finslerian setting}
\emph{Throughout, $(M^n,F)$ denotes a (forward or backward) complete Finslerian manifold and $g$ (or $g(x,\y)$) denotes the resulting Finslerian metric.} 
\par In the Finslerian setting, by the aid of the metric compatibility of the Cartan connection, one can draw parallels to the Riemannian theory. 
\par Indeed, the unit-speed geodesic circles are still given essentially by the ``same'' ODE  
\begin{align}\label{eq:geod-circ}
c'''(s) + \|c''(s)\|^2c'(s) = 0, \quad F(c'(s)) = 1,
\end{align}
where this time $' := \nabla_{\dt{\widetilde{c}}}$ in which $\widetilde{c}$ is the natural lift of $c$ to the tangent bundle $TM$ and \emph{$\nabla$ is the Cartan connection}. For a curve with a general parameter, this amounts to
\begin{align*}
\dddt{c} - 3v^{-2}g_{\dt{c}}(\dt{c}, \ddt{c})\ddt{c} \equiv 0 \mod \dt{c}, \quad v := F(\dt{c}(t)).
\end{align*}
For more details on derivation of \hyperref[eq:geod-circ]{\eqref{eq:geod-circ}} and the ODE aspects of it, see~\cite{SY}. 
\par Consequently, a Finslerian \cpt is a diffeomorphism, say $\varphi$, that preserves geodesic circles as point sets, i.e., it does not necessarily preserve the constant speed.  
\par \emph{We note that calling such maps a transformation is a bit of a misnomer (yet a nomenclature that we do not mind to use) since the inverse might not necessarily preserve circles and consequently, in general, \cpt\!s do not form a group.} 
\par What is not parallel to the Riemannian theory is the fact that a Finslerian \cpt is not necessarily conformal. We refer to \cite{SY} and \cite{FL1} for more details and examples.
\par The Finslerian PDE that describes a conformal Finslerian \cpt turns out to essentially ``parallel'' the Riemannian one. Shen-Yang~\cite{SY} proved that 
\begin{align*}
	\varphi^*F =:\bar{F} = \rho^{-1} F
\end{align*}
 is a \cpt if and only if 
\begin{align*}
\rho_{i\vert j} = \lambda g,
\end{align*}
for some scalar function $\lambda$; where, $_{\vert}$ denotes the horizontal covariant differentiation w.r.t. the Cartan (or equivalently Chern) connection. Additionally, they showed that isotropic or constant flag curvature (in $dim \ge 3$) as well as being Einstein is preserved. 
\par Note, by Knebelman's theorem, that the conformal factor can only be a scalar function, i.e., must be independent of $\y$. Furtheremore, $\bar{F} = e^{\sigma} F$ is equivalent to $\bar{g} = e^{2\sigma} g$. 
\par The coordinate-free version of the characterizing PDE with the explicit geometric RHS has been obtained in \cite{FL1} as
\begin{align}\label{eq:Fins-geom-pde}
	\nabla_{X^h} (\grad \sigma) -  d\sigma(X)\grad\sigma = f\left(x\right)X, \quad \forall X \in T_xM,
\end{align}
where $\grad \sigma$ is the gradient of $\sigma$ defined using the Legendre transform; see~\cite{FL1}.
\par The PDE has the equivalent forms
\begin{align}\label{eq:Fins-cpt-hess}
	\Hess_F \sigma - d\sigma\otimes d\sigma \circ \ident\times\ident = f\ident, \quad \text{on} \quad SM, 
\end{align} 
and
\begin{align}\label{eq:Hess-pde}
	{^{\sf horiz}\nabla}^2 \sigma - d\sigma^2 = fg,
\end{align}
for the scalar function, $f$, that takes the geometric form
\begin{align*}
	f(x) = \sigma''\restr_{\nicefrac{\y}{F(\y)}} - F(\grad \sigma)^2 + e^{2\sigma}\widebar{\kappa}^2;
\end{align*}
here, $\sigma''\restr_{\nicefrac{\y}{F(\y)}}$ here means second derivative in the direction of the unit speed geodesic with initial velocity $\nicefrac{\y}{F(\y)}$ (indeed the Hessian) and $\widebar{\kappa}$ is the curvature of this geodesic w.r.t. $\widebar{F}$; see~\cite{FL1} for more details. 
\subsection*{Classification of Finslerian \cpt\!s}
In some Finslerian spaces, the conformal \cpt\!s become very rigid. For instance, the authors established that under the following circumstances, $(M,F)$ is Riemannian: 
\begin{itemize}
	\item If in addition, $(M,F)$ is Berwaldian whose flag curvatures do not vanish on a dense subset; meaning, curvatures of all flags at every point of a dense subset are non-zero~\cite{FL1};
	\medskip
	\item If in addition, ${\sf Crit} \neq \varnothing$ and $F$ is Berwaldian or reversible~\cite{FL2}. 
\end{itemize}
\par In these notes, we intend to study the general such spaces $(M,F)$. 
\subsubsection*{\small {\bf Diffeomorphism classification and curvature rigidity}}
\hfill
 \par \dunderline{1.2pt}{\small \textit{\textbf{\textsf{Notation/Terminology/Convention:}}}}
\begin{itemize}
	\item \textsl{Throughout these notes, we suppose $(M,F)$ is a (forward or backward) complete Finslerian manifold that admits a non-trivial conformal \cpt, $e^{\sigma} F$. In addition, $\Sigma$ denotes either a portion or the whole of a regular level hypersurface of $\sigma$. We set $\rho = e^{-\sigma}$.}
	\smallskip
	\item	\textsl{We refer to flag curvatures of flags that are tangent to regular level hypersurfaces $\Sigma$ as tangential flag curvatures and denote them by $\K^{\parallel}$; similarly, flag curvatures where the flag pole is the normal direction ${\sf n}$ and the other vector is tangent to $\Sigma$ are referred to as normal flag curvatures and denoted by $\K^{\perp}$.}
	\smallskip
	\item \textsl{An osculating metric is a Riemannian metric obtained by plugging a non-vanishing vector field into the Finslerian metric; see~\hyperref[sec:geom-lin]{\secref{sec:geom-lin}} for more detials.}	
\end{itemize}
\begin{definition}
 	A Finslerian space $(M,F)$ is said to be \underline{tame} whenever
 	\begin{align*}
 		0 < \inf_{\substack{
 				x \in M\\
 				\textit{\textbf{u}},\,\y \, \in S_xM
 		}} g^{\y}({\footnotesize \textit{\textbf{u}}},{\footnotesize \textit{\textbf{u}}}) \le \sup _{\substack{
 				x \in M\\
 				\textit{\textbf{u}},\,\y \, \in S_xM
 		}} g^{\y}({\footnotesize \textit{\textbf{u}}},{\footnotesize \textit{\textbf{u}}}) < \infty.
 	\end{align*}
 \end{definition}
\begin{theorem*}\label{thrm:classification}
	Suppose $(M,F)$ is a (forward or backward) complete Finslerian manifold that admits a non-trivial conformal \cpt, $e^{\sigma} F$. Then, 
	\begin{itemize}
		\item $\sigma$ has at most two critical points.
		\smallskip
		\item All regular level hypersurfaces of $\sigma$ are diffeomorphic (denoted by $\Sigma$).
		\smallskip
		\item The curvature properties of being/having \emph{constant flag curvature, isotropic flag curvature, sectional flag curvature, Strongly (Cartan-) Einstein, Einstein, vector-indepenedent scalar curvature (Cartan or Akbarzadeh), and constant scalar curvature (Cartan or Akbarzadeh)} are all inherited by $\Sigma$ from $M$; for the definitions of various curvatures, see~\hyperref[subsec:curvature]{\secref{subsec:curvature}}. 
	\end{itemize}
Moreover, the following classification holds true. 
	\begin{enumerate}
		\item[\textrm{\small \bf Case~$\sf I$:}] \label{item:class-1} ${\sf Crit} = \varnothing$.
		\begin{enumerate}
			\smallskip
			\item \label{item:class-1-a} $M^n$ is diffeomorphic to the product $\R \times \Sigma^{n-1}$. 
			\smallskip
			\item \label{item:class-1-b} There is an osculating Riemannian metric $\mathdutchcal{g}$ such that $e^{2\sigma} \mathdutchcal{g}$ is a Riemannian \cpt. 
			\smallskip
			\item \label{item:class-1-c} If furthermore, $(M,F)$ is \underline{tame}, then $(M,\mathdutchcal{g})$ is conformally diffeomorphic to a Riemannian isometric product $\left(\R, ds^2 \right) \times \left(\Sigma, \mathdutchcal{g}_{\sf level} \right)$ where $\left(\Sigma, \mathdutchcal{g}_{\sf level} \right)$ is complete.
			\smallskip 
			\item \label{item:class-1-d} If $(M,F)$ has constant normal flag curvatures $\K^{\perp}$ and constant Catan/Akbarzadeh scalar curvature (in particular, if it is strongly (Cartan-)Einstein) equal to $\K^{\perp}$, then, $\Sigma$ has vanishing Cartan/Akbarzadeh scalar curvature (in particular, is (Cartan) Ricci flat) w.r.t. the induced Finslerian strucutre. 
		\end{enumerate}
		\medskip
		\item[\textrm{\small \bf Case~$\sf II$:}] \label{item:class-2} $\texttt{\#} \, {\sf Crit} = 1$.
		\begin{enumerate}
			\smallskip
			\item \label{item:class-2-a} $M^n$ is diffeomorphic to $\R^n$.
			\smallskip
			\item \label{item:class-2-b} The regular level sets of $\sigma$ are diffeomorphic to $\Sp^{n-1}$ and have constant positive flag curvature w.r.t the induced Finslerian structure.
			\smallskip
			\item \label{item:class-2-c} If $(M,F)$ has constant Catan or Akbarzadeh scalar curvature, then both Catan and Akbarzadeh scalar curvatures are constant and coincide; moreover, 
			\begin{align*}
			{\sf s} = \mathdutchcal{s}= \K^{\parallel} = \K^{\perp} > 0.
			\end{align*}
		\end{enumerate}
		\medskip
		\item[\textrm{\small \bf Case~$\sf III$:}] \label{item:class-3} $\texttt{\#} \, {\sf Crit} = 2$. 
		\begin{enumerate}
			\smallskip
			\item \label{item:class-3-a} $M^n$ is diffeomorphic to the sphere $\Sp^n$. 
			\smallskip
			\item \label{item:class-3-b} The regular level sets of $\sigma$ are diffeomorphic to $\Sp^{n-1}$ and have constant positive flag curvature w.r.t the induced Finslerian structure.
			\smallskip
			\item \label{item:class-3-c} If $(M,g)$ has constant Catan or Akbarzadeh scalar curvature, then both Catan and Akbarzadeh scalar curvatures are constant and coincide; additionally, 
			\begin{align*}
			{\sf s} = \mathdutchcal{s}= \K^{\parallel} = \K^{\perp} > 0.
			\end{align*}
		\end{enumerate} 
	\end{enumerate}
\end{theorem*}
\section{Preliminaries}\label{sec:prelim}
\subsection{Curvature}\label{subsec:curvature}
\subsubsection*{\small {\bf $\pmb{hh}$-curvature tensors}}
In these notes, we are concerned with the Riemannian-like curvature tensor (as opposed to purely Finslerian curvature tensors) of a Finslerian connection $\nabla$; this curvature tensor is known as the $hh$-curvature. For the Cartan connection, we denote the $hh$-curvature by $\Riem^{\y}$, and for the Chern connection, by $\mathcal{R}iem^{\y}$. 
\par In local coordinates, one sets
\begin{align*}
	\Riem(\partial_i,\partial_j)\partial_k = \Riem^{\;\;l}_{k\;\;ij}\partial_l,
\end{align*}
and similarly for $\mathcal{R}iem$. 
(compare to \cite{AIM} and \cite{BF} where they use the convention $\Riem^{\;\;l}_{k\;\;ji}\partial_l$ for this).
\par The Cartan and Chern $hh$-curvature tensors are related via
\begin{align*}
	\Riem(X,Y)Z = \mathcal{R}iem(X,Y)Z + S(X,Y)Z,
\end{align*}
where
\begin{align*}
	S(X,Y)Z =  -\nabla_{{\frak R}(X,Y)}Z.
\end{align*}
Equivalently,
\begin{align*}
	S(\partial_k,\partial_j)\partial_h =  - C^i_{hr} {\frak R}^{r}_{\;\;kj},
\end{align*}
where
\begin{align*}
	{\frak R}^k_{\;ij} = \delta_j G^k_i - \delta_i G^k_j,
\end{align*}
is the $R^1$-torsion vector; see~\cite[(2.3.2.5)]{AIM}.
\subsubsection*{\small {\bf Flag curvature}}
\par The flag curvature $\K(v\wedge w)$ is the sectional curvature ${\mathrm sec}(v,w)$ w.r.t. $g^v$ and can be computed as
\begin{align*}
	\K(\vecv \wedge \vecw) = \frac{\left<\Riem^{\vecv}\left(\vecv, \vecw\right)\vecw, \vecv \right>_{\vecv}}{\|\vecv\|_{\vecv}^2\|\vecw\|^2_{\vecv} - \left< \vecv, \vecw \right>_{\vecv}^2} = \frac{\left<\mathcal{R}iem^{\vecv}\left(\vecv, \vecw\right)\vecw, \vecv \right>_{\vecv}}{\|\vecv\|_{\vecv}^2\|\vecw\|^2_{\vecv} - \left< \vecv, \vecw \right>_{\vecv}^2},
\end{align*}
\subsubsection*{\small {\bf Ricci curvature tensors}}
\par The Ricci curvature tensor (as a symmetric tensor) has two straightforward analogues (more versions in the literature). One is
\begin{align*}
	\mathcal{R}ic_{\sf tensor} := \mathcal{R}ic_{ij} \; dx^i\otimes dx^j, \quad \mathcal{R}ic_{ij} := \nicefrac{1}{2} \, \dt{\partial}_{j} \dt{\partial}_{j} \left(\mathcal{R}ic\right),
\end{align*}
where $\mathcal{R}ic$ is Akbarzadeh's scalar Ricci,
\begin{align*}
	\mathcal{R}ic(\y) := F^2(\y)\sum\limits_{i=1}^{n-1} \K(\y\wedge e_i),
\end{align*}
in which
$
\left\{\nicefrac{\y}{F(\y)}, e_1,e_2, \cdots, e_{n-1}  \right\},
$, 
form an orthonormal basis w.r.t. $g^{\y}$; note $\mathcal{R}ic$ is independent of the choice of frame; see e.g.,~\cite[Section 5.3]{Oh}. Alternatively,  
\begin{align*}
	\mathcal{R}ic(\y) = \mathbf{R}^i_{\;i}(\y),
\end{align*}
where $\mathbf{R}$ is a natural curvature tensor that makes the Jacobi field equation hold true. The curvature tensor $\mathbf{R}$ satisfies
\begin{align*}
	\mathbf{R}^i_{\;\,j} (\y) = {\mathcal{R}iem^{\;\,i}_{l\;\;jk}}(\y) y^ky^l = {\Riem^{\;\,i}_{l\;\;jk}}(\y) y^ky^l;
\end{align*}
e.g., \cite[Section 5.4]{Oh}

\par A different Ricci tensor is obtained by taking the usual trace of $\Riem$ to obtain ${\Ric}$. This will be called the Cartan-Ricci tensor. 
\par In particular 
\begin{align*}
	\mathcal{R}ic_{\sf tensor}(\y)(\y,\y) = \mathcal{R}ic(\y) = {\Ric}(\y)(\y,\y),
\end{align*}
\par For another interesting notion of symmetric Ricci tensor obtained from the $\mathcal{R}iem$ by tracing and symmetrization, see~\cite{Li-Shen}.
\par As is by now a standard definition (originated by Akbarzadeh~\cite{AZ}), we say $(M,F)$ is \emph{Einstein} whenever
\begin{align*}
\mathcal{R}ic(c,\y) = (n-1)K(x)F^2(\y),
\end{align*}
for some function $K(x)$;i.e., when $F^{-2}\mathcal{R}ic$ is isotropic. This is equivalent to
\begin{align*}
	\mathcal{R}ic_{\sf tensor}(x,\y) = (n-1)K(x)g(x,\y);
\end{align*}
see~[Bao-Robles, page 216]. 
\par Similarly, we say $(M,F)$ is \emph{Cartan-Einstein} whenever 
\begin{align*}
	\Ric(c,\y) = (n-1)K(x)g(x,\y),
\end{align*}
for some function $K(x)$.  By the Riemannian Schur's lemma, in the Riemannian setting, both being Einstein or Cartan-Einstein reduce to the usual definition of an Einstein metric. 
\par \emph{When the function $K(x)$ is constant, we call the Finslerian metric strongly Einstein or, respectively, strongly Cartan-Einstein.} 
\subsubsection*{\small {\bf Scalar curvatures}}
\par The \emph{Cartan scalar curvature}, ${\scal}$, is the trace of ${\Ric}$. We also set
\begin{align*}
	\text{$\mathdutchcal{s}cal(x,\y)$}:= \tr \mathcal{R}ic_{\sf tensor}(x,\y),
\end{align*}
and call it the \emph{Akbarzadeh scalar curvature}.
\par In the study of constant curvature cases, the normalized (Cartan or Akbarzadeh) scalar curvatures, given by
\begin{align*}
	\textsf{s}(x,\y) := \nicefrac{1}{n(n-1)} \;\scal(x,\y), \quad \text{and} \quad \mathdutchcal{s}(x,\y) := \nicefrac{1}{n(n-1)} \; \mathdutchcal{s}cal(x,\y),
\end{align*}
come in handy.
\par When 
\begin{align*}
	{\mathsf{s}}(x,\y) = r(x),
\end{align*}
we say the scalar curvature is \emph{vector-independent} and when $r(x) \equiv r$ is cosntant, we say $F$ is of constant scalar curvature. Similar definitions also are in effect for $\mathdutchcal{s}cal$. 
\subsection{$Y$-calculus}\label{subsec:Y-calc}
Let $Y = Y^i\partial_i$ be a local non-vanishing vector field. Let us denote the resulting Riemannian metric $g^{Y}$ by $\mathdutchcal{g}$. The metric  $\mathdutchcal{g}$ is also referred to as an osculating Riemannian metric. 
\par By direct computation using the chain rule, the Christoffel symbols of the Levi-Civita connection of ${\mathdutchcal{g}}$ satisfy
\begin{align*}
	{^{\mathdutchcal{g}}\Gamma}^i_{jk}(x) &= \overstar{\Gamma}^i_{jk}(Y) \\
	&+ g^{il} \Big( C_{lkm} \left( G^m_j + \partial_{j} Y^m  \right) - C_{ljm} \left(  G^m_k - \partial_{k} Y^m  \right) + C_{jkm} \left( G^m_ l - \partial_{l} Y^m  \right) \Big)(Y);
\end{align*}
e.g., see~\cite[Chapter 4]{Oh}.
This connection is obviously ${\mathdutchcal{g}}$-compatible, but it does not match Cartan's covariant differentiation along curves (when $Y$ is an extension of the tangent field of the curve).
\par The Barthel-Matsumoto affine connection $^Y\bbGamma$ (or Cartan $Y$-connection) is another ${\mathdutchcal {g}}$-compatible affine connection on the bundle $TM$ given by Christoffel symbols
\begin{align*}
	^Y\bbGamma^i_{jk}(x) := \overstar{\Gamma}^i_{jk}(Y) + C^i_{jm}(Y) \left( G^m_k(Y) + \partial_{k} Y^m \right).
\end{align*}
\par It turns out that along a curve $c$, if $Y$ is chosen to be an extension of $\dt{c}$, the covariant differentiation w.r.t. the Barthel-Matsumoto connection coincides with that if Cartan's connection in the sense that,
\begin{align*}
	^Y{\mathbb{D}}_{\dt{c}}U = \nabla_{\dt{\tilde{c}}} (U),
\end{align*}
holds for all vector fields $U$ along $c(t)$. Here, $\widetilde{c}$ - a curve in $TM$ - is the natural lift of $c$; thus, it satisfies
\begin{align*}
	\dt{\widetilde{c}}= \dt{c}^l\delta_l + (\ddt{c}^{\,l} + 2G^l)\dt{\partial}_l.
\end{align*}
See~\cite{Mat2} and \cite[2.6.3]{AIM} for more details on the Barthel-Matsumoto connection. 
\section{Geometry of the underlying space}\label{sec:highlights}
\par In this section, we briefly highlight some recent developments about the geometry $(M,F)$ by the authors, presented in \cite{FL1} and \cite{FL2}. 
\par Around a regular point of $\sigma$, we consider the set of coordinates $\left\{u^1=u, u^2,\cdots, u^n\right\}$ where $u$ is the arclength of the normal geodesics, i.e., the coordinate $u$ is the one that integrates 
\begin{align*}
	\partial_u := \nicefrac{\grad^{\y} \sigma}{F(\grad^{\y} \sigma)} = \nicefrac{\grad \sigma}{F(\grad \sigma)},
\end{align*}
and $u^2,\cdots, u^{n-1}$ are a set of coordinates on $\Sigma$. 
\par \emph{
	The flow of the vector field $\grad \sigma$ provides a diffeomorphism between the regular level sets of $\sigma$ as long as it does not encounter a critical point of $\sigma$. As a result, any chosen local coordinate system $(u^2, \cdots, u^n)$ for an open domain $U \subset \Sigma_{t}$ at time $t \in (s_1,s_2)$ can be extended to a coordinate system $u^1=u, u^2, \cdots, u^n$ on the open domain $(s_1,s_2)\times \Sigma_t$ given that $\sigma^{-1}\left( [s_1,s_2] \right) \cap {\sf Crit} = \varnothing$}.
\par \dunderline{1.2pt}{\small \textit{\textbf{\textsf{Notation:}}}}
\begin{itemize}
\item \textsl{Throughout, the small case Greek letters signify the indices $2, \cdots, n$.}
\smallskip
\item \textsl{The unit vector $\nicefrac{\grad\sigma}{\|\grad \sigma\|}$ on the regular domain, is denoted by ${\sf n}$ and forms a geodesic unit vector field.}
\end{itemize}
\subsection{The regular domain}
On the regular domain of $\sigma$ and in coordinates $(u, u^\alpha)$, the following hold true.
 \begin{enumerate}[label=\textbf{\footnotesize (Reg\arabic*)}]
 	\item \label{item:highlight-1}
 	 $f = \sigma_{uu} - \sigma_u^2 = -\nicefrac{\rho_{uu}}{\rho}$.
 	\medskip
 	\item[] For more details and the proofs, see~\cite{FL1}. 
 	\medskip
 	\item \label{item:highlight-2}
 	For $\vecv \neq 0$, we have
 	\begin{align}\label{eq:metric-decomp}
 		g(x,\y) = du^2 + (e^{-\sigma}\sigma_u)^2\; {g}_{\sf ideal}(x,\vecv),\quad \vecv = (y^1,\cdots,y^{n-1}),
 	\end{align}
 	for a $(u^1,y^1)$-independent Finslerian metric ${g}_{\sf ideal}$ on $\Sigma$. Equivalently,
 	\begin{align*}
 		F^2 = |\cdot|^2 + (e^{-\sigma}\sigma_u)^2 F_{\sf ideal}^2,
 	\end{align*}
 	holds true where $|\cdot|$ is the standard fundamental function on $I$ and $F_{\sf ideal}$ is a constant Finslerian fundamental function on $\Sigma$. 
 	\par  In particular,
 	\begin{align}\label{eq:warped}
 		g^{\sf n} = du^2 + (e^{-\sigma}\sigma_u)^2 \mathdutchcal{g}_{\sf ideal} = du^2 + \left(\rho'\right)^2 \mathdutchcal{g}_{\sf ideal}.
 	\end{align}
 	for a $u$-independent Riemannian metric $\mathdutchcal{g}_{\sf ideal}$ on the regular hypersurface $\Sigma$.
 	\medskip 
 	\item \label{item:highlight-3} The induced connection from the Cartan connection on the level hypersurfaces $\Sigma$ coincides with the Cartan connection of the induced Finslerian metric. Moreover, the same holds for the Chern connection.  
 	\medskip
 \par \dunderline{1.2pt}{\small \textit{\textbf{\textsf{Notation/Terminology:}}}}
 \textsl{In the sequel, the index ``${\sf ideal}$'' indicates the use of the metric $g_{\sf ideal}$ (which is $(u,y^1)$-independent) on $\Sigma$ while the index ``${\sf level}$'' indicates the use of $g_{\sf level} := \rho^2_u g_{\sf ideal}$. Moreover, $\Sigma$ equipped with $g_{\sf ideal}$ (or sometimes as a smooth manifold entity) will be called the ideal hypersurface.}
  \medskip
 	\item \label{item:highlight-4} Setting
 	\begin{align*}
 		\mathcal{L}:= \left(\nicefrac{\rho_{uu}}{\rho_u}\right)^2, \quad \K^{\perp} := - \nicefrac{\rho_{uuu}}{\rho_u} = \left(f -\; \nicefrac{df(\grad\sigma)}{\|\grad\sigma\|^2}\right),
 	\end{align*}
 the following identities hold for the $hh$-curvature tnesor: 
 	\begin{enumerate}
 		\medskip
 		\item
 		${^{\sf level}\Riem}^{\vecv}(X,Y)Z  = {^{\sf ideal}\Riem}^{\vecv}(X,Y)Z $.
 		\medskip
 		\item \label{item:curv-1}
 		\hspace{-29pt}
 		\makebox[\linewidth]{\(\begin{aligned}[t]
 				{\Riem}^{\y}(X,Y)Z &= {^{\sf ideal}\Riem}^{\vecv}(X,Y)Z - \mathcal{L} \Big( \left< Y,Z\right>_{\!\vecv}X - \left<X,Z \right>_{\!\vecv}Y \Big)\\
 				&= {^{\sf ideal}\Riem}^{\vecv}(X,Y)Z - \rho^2_{uu} \Big( \left< Y,Z\right>^{\sf ideal}_{\!\vecv}X - \left<X,Z \right>^{\sf ideal}_{\!\vecv}Y \Big).
 			\end{aligned}\)} 
 		\medskip
 		\item \label{item:curv-2}
 		$
 		\Riem^{\y}(X,Y){\sf n} = 0
 		$.
 		\medskip
 		\item \label{item:curv-3}
 		$
 		\Riem^{\y}(X,{\sf n}){\sf n} = (f -\; \nicefrac{df(\grad\sigma)}{\|\grad\sigma\|^2})X = \K^{\perp} X
 		$.
 		\medskip
 		\item \label{item:curv-4}
 		$
 		\Riem^{\y}(X,{\sf n})Y =  -\K^{\perp} \left< X,Y \right>_{\!\vecv} {\sf n}
 		$.
 		\medskip
 		\item \label{item:curv-5}
 		
 		\hspace{-90pt}\makebox[\linewidth]{\(\begin{aligned}[t]
 					\K\left( \y \wedge Y \right) = \K\left( \y \wedge Y \right) = {\K}_{\sf level}(\vecv \wedge Y) - \mathcal{L},
 			\end{aligned}\)} \hfill
 		\newline in particular,
 		\begin{align*}
 			\K\left( X \wedge Y \right) = {\K}_{\sf level}(X \wedge Y) - \mathcal{L}.
 		\end{align*}
 		\item \label{item:curv-6}
 		$
 		\K\left( {\sf n} \wedge X \right) = \K^{\perp}
 		$.
 	\end{enumerate}
 	\medskip
 	\item \label{item:highlight-5} Verbatim statements hold for the Chern curvature tensor $\mathcal{R}iem$.
 	\medskip	
 	\item \label{item:highlight-6} $\Riem$ and $\mathcal{R}iem$ are independent of $y^1$ when $\vecv \neq 0$. 
 	\end{enumerate}
 	For more details and the proofs, see~\cite{FL2}. 
 	\subsection{Near critical points}
 	\par Regarding the critical points and the geometry in close vicinity thereof, the following were established in \cite{FL2}. 
 	\begin{enumerate}[label=\textbf{\footnotesize (Crit\arabic*)}]
 	\item \label{item:const-curv-1} Critical points of $\sigma$ are isolated.
 	\smallskip
 	\item \label{item:const-curv-2}
 	Near each critical point $p$, the level sets of the forward or backward distance function from $p$ (i.e., forward or backward geodesic spheres) coincide with level sets of $\sigma$. Therefore, we can take $u$ to be ($\pm$) the forward or backward distance function from $p$, $r = \dist_p$; i.e., $r = |u|$. This in particular means the forward and backward distances from $p$ coincide (at least near $p$). 
 	\smallskip
	\item \label{item:const-curv-3} The quantities 
	\begin{align*}
	\rho_u(0) := \lim_{r\downarrow 0} \rho_u(u), \quad \rho_{uu}(0) := \lim_{r\downarrow 0} \rho_{uu}(u), \quad \rho_{uuu}(0) := \lim_{r\downarrow 0} \rho_{uuu}(u), 
	\end{align*}
	exist, and we obtain
	\begin{align*}
	\rho_u(0) = 0, \quad \rho_{uu}(0) \neq 0, \quad \text{and} \quad \rho_{uuu}(0) = 0,
	\end{align*}
	Equivalently, the same statement holds for $\sigma$.
	\smallskip
	\item \label{item:const-curv-4} On a deleted neighborhood of a critical point $p$, the ideal metric $g_{\sf ideal}$ has constant positive flag curvature $\rho^2_{uu}(0)$ and the ideal hypersurface is diffeomorphic to $\Sp^{n-1}$. 
	\smallskip
	\item \label{item:const-curv-5} The critical points are non-degenerate; thus, they are local extrema of $\sigma$.
	\end{enumerate}
\section{Geometric linearization}\label{sec:geom-lin}
Recall in the Riemannian setting, the PDE describing a \cpt is independent of $\y$ and a \cpt, $\varphi$, is automatically conformal. The geometric characterization of local regular Riemannian \cpt\!s is as follows.
\begin{enumerate}
	\item \label{item:geom-char-1} The equation \hyperref[eq:Riem-pde]{\eqref{eq:Riem-pde}} admits a local regular solution if and only if locally the space is foliated by totally umbilic hypersurfaces whose normals are principal Ricci directions and integrate to geodesics~\cite{Yano2}
	\medskip
	\item \label{item:geom-char-2} The equation  \hyperref[eq:Riem-pde]{\eqref{eq:Riem-pde}} admits a local regular solution if and only if locally the metric has a warped product form with warping function $e^{-\sigma}\sigma_u$ with cross sections that are totally umbilical;~\cite{Br,Fia,Tash1}.
\end{enumerate}
\subsection{Geometric linearizations}
\par Recall for a unit vector field $V$, \emph{the Riemannian metric $g^V$ is called an \underline{osculating metric}}; see~\cite[Chapter 5]{Oh} and the references given there.  
\par In these notes, a geometric linearization of a Finslerian geometric PDE is referred to as finding osculating metrics that are Riemannian \cpt\!s. More precisely,
\begin{definition}
	An osculating Riemannian metric $g^V$ is said to be a local regular geometric linearization of the Finslerian equation \hyperref[eq:Fins-geom-pde]{\eqref{eq:Fins-geom-pde}} (or equivalently \hyperref[eq:Hess-pde]{\eqref{eq:Hess-pde}}), if it is a local regular Riemannian \cpt (i.e., a local regular solution of \hyperref[eq:Riem-pde]{\eqref{eq:Riem-pde}}). A global Riemannian \cpt\!, $e^{2\sigma}\mathdutchcal{g}$, is said to be global linearization of $g$ whenever, restricted to the regular domain of $\sigma$, $\mathdutchcal{g}$ is an osculating metric. 
\end{definition}
\begin{theorem}[Local regular linearizations of a Finslerian conformal \cpt]\label{thrm:Riem-linearization}
	Let $\widebar{F} = e^{\sigma} F$ be a non-trivial conformal \cpt\!. Then the following are equivalent
	\begin{enumerate}
		\item The conformal transformation $\widebar{g}^V = e^{2\sigma}g^V$ is a local regular geometric linearization of the Finslerian \cpt equation \hyperref[eq:Fins-geom-pde]{\eqref{eq:Fins-geom-pde}}.
		\medskip
		\item In the coordinate system $u^1= u, \cdots, u^n$ (described in \hyperref[sec:highlights]{\secref{sec:highlights}}), it holds
		\begin{align}\label{eq:lin-criterion}
			C_{\alpha \beta \eta}\; \partial_{u} V^\eta = C_{\alpha \beta r} \; \partial_{u} V^r = 0, \quad 1 \le \alpha,\beta \le n-1,
		\end{align} 
		(evaluated at $V$)
		or equivalently
		\begin{align*}
			^V\bbGamma^i_{j1}(x) = \overstar{\Gamma}^i_{j1}(V), \quad \forall i,j.
		\end{align*}
		or
		\begin{align*}
			{\overstar{\Gamma}}^1_{ij}(V) = {\Gamma}^1_{ij}(V) = {^{\mathdutchcal{g}}\Gamma}^1_{ij}, \quad \forall i,j. 
		\end{align*}
	\end{enumerate}	
	\par Moreover, for any such Riemannian \cpt\!, $g^V$, the principal curvature of $\Sigma$ w.r.t. $g$, $g^{\sf n}$ and $g^V$ all coincide and equal
	$\kappa = \nicefrac{f}{F(\grad \sigma)}$.
	\par Also note that any vector field, $V$, that is independent of $u$, satisfies the linearization criterion.
	\par In particular,	$e^{2\sigma}g^{\sf n}$ is (well-defined and) a geometric linearization on the regular domain of $\sigma$.
\end{theorem}
\begin{proof}[\small \textbf{Proof}]
	Based on the decomposition \hyperref[eq:metric-decomp]{\eqref{eq:metric-decomp}}, we have
	\begin{align}\label{eq:met-decomp-V}
		g_{\alpha\beta}^V &= du^2 + \left(e^{-\sigma}\sigma_u\right)^2 g_{\alpha\beta}\left(V^2, \cdots, V^n\right).
	\end{align}
	This readily implies the normal geodesics are still unit-speed geodesics w.r.t. $g^V$. By virtue of the geometric characterization as in \hyperref[item:geom-char-2]{item (\ref{item:geom-char-2})} in the beginning of this section, and given the decomposition \hyperref[eq:met-decomp-V]{\eqref{eq:met-decomp-V}}, $g^V$ is a local regular Riemannian \cpt, if and only if,
	\begin{align*}
		h_{\alpha\beta} := g_{\alpha\beta}\left(V^2, \cdots, V^n\right),
	\end{align*}
	is a constant metric on $\Sigma$, i.e., independent of $u$. 
	\par By the chain rule, we get
	\begin{align*}
		\partial_u h_{\alpha\beta} &= \partial_u g_{\alpha\beta}\left(V^2, \cdots, V^n\right)\\
		&= \dt{\partial}_\kappa g_{\alpha\beta} \; \partial_u V^\kappa \\
		&= C_{\alpha\beta \kappa} \; \partial_u V^\kappa;
	\end{align*}
	thus, $g^V$ is a local regular Riemannian \cpt if and only if 
	\begin{align}\label{eq:criterion-1}
		C_{\alpha\beta \kappa} \; \partial_u V^\kappa  \equiv 0.
	\end{align}
	\par Since $C_{1ij} \equiv 0$ (see~\cite[Theorem 3.4]{FL1}), we get
	\begin{align*}
	C^i_{jm}\; \partial_u V^m &= C^i_{j\gamma} \; \partial_u V^\gamma\\
		&= g^{il}C_{lj\gamma} \; \partial_u V^\gamma\\
	&= g^{i\eta}C_{\eta j\gamma}  \; \partial_u V^\gamma\\
	&= 0.
	\end{align*}
	Thus, noting $G^i_1 \equiv 0$ (established in~\cite[Theorem 3.4]{FL1}), \hyperref[eq:criterion-1]{\eqref{eq:criterion-1}} is equivalent to 
	\begin{align*}
		C^i_{jm}(V) \left(	G^m_1 + \partial_u V^m\right)(V) \equiv 0,
	\end{align*}
	which in turn is equivalent to 
	\begin{align*}
		^V\bbGamma^i_{j1}(x) = \overstar{\Gamma}^i_{j1}(V);
	\end{align*}
	see \hyperref[subsec:Y-calc]{\secref{subsec:Y-calc}}.
	\par On the other hand, by the standard relation,
	\begin{align*}
		\overstar{\Gamma}^k_{ij} - \Gamma^k_{ij} =  g^{kl}\left(C_{ijr}G^r_l - C_{ljr}G^r_i - C_{lir}G^r_j\right),
	\end{align*}
	(see~\cite[2.4.9]{BCS}), we get
	\begin{align*}
		\overstar{{\Gamma}}^1_{\alpha\beta} - {\Gamma}^1_{\alpha\beta} = {C}_{\alpha\beta r}{G}^r_1 = {C}_{\alpha\beta \gamma}{G}^\gamma_1 = 0.
	\end{align*}
	By direct calculations, one gets
	\begin{align*}
		{^{\mathdutchcal{g}}\Gamma}^i_{jk}(V) &= {\Gamma}^i_{jk}(V) \\
		&+ {g}^{il} \Big({C}_{lkm} \; \partial_{j} V^m  +  {C}_{jlm} \; \partial_{k} V^m  - {C}_{jkm} \; \partial_{l} V^m \Big)(V);
	\end{align*}
	see e.g., \cite[(4.1)]{Oh}.
	Thus,
	\begin{align*}
		&{^{\mathdutchcal{g}}\Gamma}^1_{\alpha\beta} - {\Gamma}^1_{\alpha\beta}(V) \\
		&=   \delta^{1l}  \Big({C}_{l\beta m} \; \partial_{\alpha} V^m +  {C}_{\alpha lm} \; \partial_{\beta} V^m - {C}_{\alpha \beta m} \; \partial_{l} V^m \Big)(V)\\
		&=  \Big({C}_{1\beta m} \; \partial_{\alpha} V^m +  {C}_{\alpha1m} \; \partial_{\beta} V^m - {C}_{\alpha \beta m} \; \partial_{u} V^m  \Big)(V)\\
		&= -  {C}_{\alpha \beta m}(V) \; \partial_{u} V^m(V) \\
		&= -  {C}_{\alpha \beta \gamma}(V)  \; \partial_{u} V^\gamma(V) \\
		&=0,
	\end{align*}
	where the last equality is by \hyperref[eq:criterion-1]{\eqref{eq:criterion-1}}. 
\par Therefore, one deduces that \hyperref[eq:criterion-1]{\eqref{eq:criterion-1}} is also equivalent to
	\begin{align*}
		{\overstar{\Gamma}}^1_{\alpha\beta}(V) = {\Gamma}^1_{\alpha\beta}(V) = {^{\mathdutchcal{g}}\Gamma}^1_{\alpha\beta}.
	\end{align*}
	We note that the other Christoffel symboles $\Gamma^1_{ij}$ or ${\overstar{\Gamma}}^1_{ij}$ vanish; see~\cite{FL1}.
\par To show that $e^{2\sigma}g^{\sf n}$ is (well-defined and) a geometric linearization on the regular domain of $\sigma$, we note that since the vector field ${\sf n}$ is constant in coordinates, this follows directly from the criterion \hyperref[eq:lin-criterion]{\eqref{eq:lin-criterion}} or alternatively from the combination of \hyperref[eq:warped]{\eqref{eq:warped}} and the geometric characterization, \hyperref[item:geom-char-2]{item (\ref{item:geom-char-2})} at the beginning of the current section.
\end{proof}
\par A global converse to the Riemannian linearization phenomenon is also available in special cases that include Berwaldian manifolds. 
\begin{theorem}[Non-linearization]\label{thrm:Fins-nonlinearization}
	Let $\widebar{F} = e^{\sigma} F$ be a non-trivial conformal transformation and suppose on the regular domain of $\sigma$, $F$ satisfies 
	\begin{align*}
		G^1(x,\y) \equiv {^{g^{\sf n}}G}^1(x,\y).
	\end{align*}
(In particular, if $(M,F)$ is Berwaldian). Also suppose $\bar{g}^{\sf n} = e^{2\sigma} g^{\sf n}$ is a local regular Riemannian \cpt\! that extends to a global metric $\mathdutchcal{g}$ such that $\widebar{\mathdutchcal{g}} := e^{2\sigma} \mathdutchcal{g}$ is a global Riemannian \cpt\!. 
	\par If either $\mathdutchcal{g}$ is complete or ${\sf Crit}$ is a nowhere dense subset, then 
	 $\widebar{F} = e^{\sigma} F$ is a Finslerian \cpt\!. 
\end{theorem}
\begin{proof}[\small \textbf{Proof}]
	At the regular points of $\sigma$, one can define two Hessians. One is the Finslerian Hessian $\Hess_F(\sigma)$ and the other is the Hessian w.r.t. $g^{\sf n}$ i.e., ${^{g^{\sf n}}\nabla}^2(\sigma)$. (Note that $g^{\sf n}$ is nothing but $g^{\grad \sigma}$.)
	\par Using the local formula for the Hessian, straightforward calculation shows that
	\begin{align}\label{eq:hess-coincide}
		{^{g^{\sf n}}\nabla}^2(\sigma) (\y) - \Hess_F \sigma (\y) &= 2\Big( G^i(y) - {^{g^{\sf n}}G}^i(y) \Big) \sigma_i\\
		&= 2\Big( G^1(y) - {^{g^{\sf n}}G}^1(y) \Big) \sigma_u \notag\\
		&= 0.\notag
	\end{align}
	\par From \hyperref[eq:hess-coincide]{\eqref{eq:hess-coincide}} in combination with \hyperref[eq:Fins-cpt-hess]{\eqref{eq:Fins-cpt-hess}}, we infer that the Riemannian PDE \hyperref[eq:Riem-pde]{\eqref{eq:Riem-pde}} coincides with the Finslerian $\cpt$ PDE \hyperref[eq:Fins-cpt-hess]{\eqref{eq:Fins-cpt-hess}} on the regular domain of $\sigma$. Now if $\mathdutchcal{g}$ is complete, one deduces that ${\sf Crit}$ is discrete (indeed by the Riemannian theory, ${\sf crit}$ cannot have more than two points) so in any case, the regular domain is dense thus by continuity, the PDE \hyperref[eq:Fins-cpt-hess]{\eqref{eq:Fins-cpt-hess}} is satisfied globally. 
	Therefore, we have concluded that $\widebar{F} = e^{\sigma} F$ is a conformal Finslerian \cpt. 
	\par Now we show that the theorem can be used in the setting of Berwaldian manifolds. 
	\par {\small \bf Claim.} When $F$ is Berwaldian, the two Hessians coincide.  
	\par {\small \bf Proof of the claim:}
	We note that the difference of Hessians
	\begin{align*}
		{^{g^{\sf n}}\nabla}^2(\sigma) (\y) - \Hess_F \sigma (\y) = 2\Big( G^i(y) - {^{g^{\sf n}}G}^i(y) \Big) \sigma_i,
	\end{align*}
	also coincides with ${\bf T}_{\grad \sigma} (y)$ where ${\bf T}_v(w)$ is the so-called ${\bf T}$-curvature; see~\cite[Lemma 14.1.1]{Shen-Lectures}. 
	\par The claim follows by utilizing the fact that the underlying space being Berwaldian is equivalent to ${\bf T}$ curvature vanishing identically. Indeed, ${\bf T}\equiv 0$ is equivalent to ${\bf P} \equiv 0$ where $\bf P$ is the $\bf P$-curvature of the Chern connection, and this in turn is equivalent to metric being Berwaldian; e.g., see~\cite[Propositions 10.1.1 and 7.2.2]{Shen-Lectures}. 
	\scalebox{0.6}{$\blacksquare$}
\end{proof}
\begin{remark}
	The ${\bf T}$-curvature satisfies ${\bf T}_{\y}(\y) = 0$ thus, in general, we always have
	\begin{align*}
	{^{g^{\sf n}}\nabla}^2(\sigma) ({\sf n}) = \Hess_F (\sigma) ({\sf n});
	\end{align*}
	see~\cite{Shen-Lectures}. 
\end{remark}
\section{Diffeomorphic classification and curvature}
\begin{lemma}\label{lemma:crit}
	The following hold true. 
	\begin{enumerate}
		\item \label{item:level-diff-1} $\sigma$ has at most two critical points, and each critical point (if any) is either an absolute maximum or absolute minimum point for $\sigma$. 
		\medskip
		\item \label{item:level-diff-2}	$M \smallsetminus \mathsf{Crit}$ is diffeomorphic to $I \times \Sigma$ for an open interval $I$. When $\mathsf{Crit} = \varnothing$, one obtains $M \cong \R \times \Sigma$.  
	\end{enumerate}
\end{lemma}
\begin{proof}[\small \textbf{Proof}]
	\hfill
	\begin{enumerate}
		\item []{\bf{\footnotesize \textsf{Proof of \hyperref[item:level-diff-1]{(\ref{item:level-diff-1})}}}}
		\par Let $q \in M \smallsetminus {\sf Crit}$. Take a small open neighborhood of $q$ of the form
		\begin{align*}
			\mathcal{U}_q = (\sigma(q)-\epsilon, \sigma(q) + \epsilon) \times \Uplambda_q,
		\end{align*}
		where $\Uplambda_q \subset \sigma^{-1}(q)$ is an open neighborhood of $q$. 
		Note that for sufficiently small $\eps >0$, $\mathcal{U}_q$ is sans critical points. 
		\par In the open neighborhood $\mathcal{U}_q$, we have the decomposition
		\begin{align*}
			g = du^2 + (\rho')^2 g_{\sf ideal}.
		\end{align*}
		This decomposition can be extended on the entire $(\sigma(q)-\epsilon, \sigma(q) + \epsilon) \times \Sigma_q$; see~\hyperref[sec:highlights]{\secref{sec:highlights}}. So w.l.o.g., we can take
		\begin{align*}
			\mathcal{U}_q = (\sigma(q)-\epsilon, \sigma(q) + \epsilon) \times \Sigma_q.
		\end{align*}
		\par Let
		\begin{align*}
			\mathcal{A} := \left\{ (s_1,s_2) \; \text{\textbrokenbar} \; s_1 < s_2 \quad \text{and \hyperref[eq:metric-decomp]{\eqref{eq:metric-decomp}} holds on}\; (s_1,s_2) \times \Sigma_q\right\} \neq \varnothing,
		\end{align*}
		and define
		\begin{align*}
			a_{\infty} := \inf \proj_1(\mathcal{A}), \quad b_{\infty} := \sup \proj_2(\mathcal{A}). 
		\end{align*}
		\par Let us discuss the possible cases.
		\begin{enumerate}
			\medskip
			\item \label{item:case-a} $-\infty < a_\infty < b_{\infty} < \infty$. 
			In this case, the normal geodesic $\gamma$ through $q$ has a limit point $q_1$ with $\sigma(q_1) = a_\infty$ and a limit point $q_2$ with $\sigma(q_2) = b_\infty$. These two points must be critical; otherwise, by the preceding argument, the metric decomposition can be extended past $a_\infty$ and $b_{\infty}$ and that is a contradiction. 
			\par It follows from \hyperref[sec:highlights]{\secref{sec:highlights}} - \hyperref[item:const-curv-2]{\ref{item:const-curv-2}}  that the level hyperspaces $\Sigma_s = \sigma^{-1}(s)$ are diffeomorphic to $\Sp^{n-1}$ for $s\in (a_\infty , b_{\infty})$; see~\cite[the proof of Lemma 5.1]{FL2}.
			\par By \hyperref[sec:highlights]{\secref{sec:highlights}} - \hyperref[item:const-curv-5]{\ref{item:const-curv-5}}, the points $q_i$ are local extrema for $\sigma$; in particular, they are elliptic points. This means adding the points $q_1$ and $q_2$ to $(a_\infty, b_\infty) \times \Sigma_q$, we obtain a smooth compact manifold $M'$.
			\par Now we note that, no other critical point can exist; indeed, since critical points are non-degenerate, a third critical point must be included in $(a_\infty, b_\infty) \times \Sigma_q$; otherwise, either $B_\eps (q_1) \smallsetminus \{q_1\}$ or $B_\eps (q_2) \smallsetminus \{q_2\}$ would be disconnected for small values of $\epsilon$ which is a contradiction. Thus, we conclude $M' = M$.
			\medskip  
			\item \label{item:case-b} $ a_\infty = -\infty$ and $  b_{\infty} < \infty$. 
			In this case, adding $q_2$ to $(a_\infty, b_\infty) \times \Sigma_q$, we get a (forward or backward) complete Finslerian manifold $M'$. There exists no other critical point by a similar connectivity argument. Thus $M'=M$.  In particular, all regular level hypersurfaces are diffeomorphic.
			\medskip 
			\item\label{item:case-c}  $ a_\infty > -\infty$ and $  b_{\infty} = \infty$.
			This case is similar to the \hyperref[item:case-b]{item (b)} above (with obvious necessary modifications).
			\medskip
			\item \label{item:case-d} $a = -\infty$ and $b = +\infty$. In this case, clearly, there exist no critical points. The flow of $\grad \sigma$ provides diffeomorphisms between all level hypersurfaces. 
		\par \emph{Now since the critical points are nondegenrate and there exist at most two of them, standard analysis facts about complete domains, yield the critical points are indeed absolute extrema.} 
		\end{enumerate}
		\medskip
		\item []{\bf{\footnotesize \textsf{Proof of \hyperref[item:level-diff-2]{(\ref{item:level-diff-2})}}}}
	\par By \hyperref[item:level-diff-1]{item (\ref{item:level-diff-1})} above, there are at most two critical points that are non-degenerate global extrema. The integral flow curves of $\grad\sigma$ coincide with normal geodesics, which are reversible by~\cite[Theorem 3.4]{FL1}. Using (forward or backward) completeness and using the Finslerian Hopf-Rinow theorem (see e.g., \cite[Theorem 3.21]{Oh}), the vector field $\grad \sigma$ is a complete vector field. Consequently, on $M \smallsetminus {\sf Crit}$, the flow of this vector field provides the diffeomorphic splitting.  	
		\end{enumerate}		
\end{proof}
\par \dunderline{1.2pt}{\small \textit{\textbf{\textsf{Notation:}}}}
\begin{itemize}
	\item Below, $X,Y,Z$ denote vector fields that are tangent to the level hypersurfaces $\Sigma$.
	\smallskip
	\item In the $u^i$ coordinates, we set $\y = (y^1, \vecv)$ where $\vecv = (y^2, \cdots, y^n)$ is also identified with $(0,y^2,\cdots,y^n)$.
\end{itemize}
\begin{lemma}[Induced Ricci and scalar curvatures on $\Sigma$]\label{lem:curvature}
	Setting
	\begin{align*}
	\mathcal{L}:= \left(\nicefrac{\rho_{uu}}{\rho_u}\right)^2, \quad \text{and} \quad \K^{\perp} = - \nicefrac{\rho_{uuu}}{\rho_u} = \left(f -\; \nicefrac{df(\grad\sigma)}{\|\grad\sigma\|^2}\right),
	\end{align*}
	(recall $\K^{\perp}$ is the normal flag curvature; see~\ref{sec:highlights}), the following identities hold for Ricci and scalar curvatures (as long as $\vecv \neq 0$).
	\begin{enumerate}
		\medskip
		\item \label{item:ric-curv-1}
	\hspace{-56pt}
	\makebox[\linewidth]{\(\begin{aligned}[t]
				{\Ric}^{\y}(X,Y) &= {^{\sf level}\Ric}^{\!\vecv}(X,Y) -\Big( (n-2) \mathcal{L}   - \K^{\perp}   \Big) \left< X,Y \right>_{\!\vecv}\\
			&= {^{\sf ideal}\Ric}^{\!\vecv}(X,Y) - \Big( (n-2) \rho^2_{uu}   - \rho_u\rho_{uuu}   \Big)\left< X,Y \right>^{\sf ideal}_{\!\vecv}.
		\end{aligned}\)} 
	\medskip
		\item \label{item:ric-curv-2}
		$
		{\Ric}^{\y}({\sf n},{\sf n}) = (n-1)\K^{\perp}
		$.
		\medskip
		\item \label{item:ric-curv-3}
		$
		{\Ric}^{\y}(X,{\sf n}) = 0
		$.
		\medskip
		\item \label{item:ric-curv-4}
		\hspace{-82pt}
		\makebox[\linewidth]{\(\begin{aligned}[t]
				\mathcal{R}ic(\y) &= {^{\sf level}\mathcal{R}ic}(\vecv) - \left( (n-2)\mathcal{L} - \K^{\perp} \right)F^2(\vecv)\\
				&= {^{\sf ideal}\mathcal{R}ic}(\vecv) - \Big( (n-2) \rho^2_{uu}   - \rho_u\rho_{uuu}   \Big)F_{\sf ideal}^2(\vecv).
			\end{aligned}\)} 
		\medskip
		\item \label{item:ric-curv-5}
		Items \hyperref[item:ric-curv-1]{(\ref{item:ric-curv-1})} -- \hyperref[item:ric-curv-3]{(\ref{item:ric-curv-3})} hold verbatim with $\mathcal{R}ic_{\sf tensor}$ in place of $\Ric$. 
		\medskip
			\item \label{item:ric-curv-6}
			For the scalar curvature, for $\vecv \neq 0$, it holds
			\begin{align*}
			\scal(\y) &= 2(n-1)\K^{\perp} + {\scal}_{{\sf level}}(\vecv) - (n-1)(n-2)\; \mathcal{L},
			\end{align*}
			and ${\scal}_{\sf level}(\vecv) = \rho^{-2}_u \; {\scal}_{\sf ideal}(\vecv)$.
			\par In other terms, the normalized Cartan scalar curvature satisfies
			\begin{align}\label{eq:scalar-normalized}
				\mathsf{s} &= \nicefrac{2}{n}\; \K^{\perp} + \nicefrac{(n-2)}{n} \; \mathsf{s}_{\sf level} \, - \, \nicefrac{(n-2)}{n} \;\mathcal{L},
			\end{align}
			and ${\mathsf{s}}_{\sf level} = \rho_u^{-2} \; {\mathsf{s}}_{\sf ideal}$.
			\item \label{item:ric-curv-7}
			The same scalar curvature relations hold with $\mathdutchcal{s}cal$ in place of $\scal$ and $\mathdutchcal{s}$ in place of $\mathsf{s}$.
\end{enumerate}
\end{lemma}
\begin{proof}[\small \textbf{Proof}]
	\hfill
	\begin{enumerate}
	\item	[]{\bf {\footnotesize \textsf{Proof of \hyperref[item:ric-curv-1]{(\ref{item:ric-curv-1})}}}.}
	\par Let 
	\begin{align*}
		\left\{{\sf n}, E_2(\y), \cdots, E_{n}(\y)\right\},
	\end{align*}
be a $\y$-dependent $g$-orthonormal frame. Clearly, $E_i$ must be tangent to $\Sigma$; thus, we deduce that $\left\{E_i({\y})\right\}$ is a $g_{\sf level}$-orthonormal frame as well. Consequently, utilizing \hyperref[sec:highlights]{\secref{sec:highlights}} - \hyperref[item:curv-1]{\ref{item:curv-1}}, we obtain
	\begin{align*}
		{\Ric}^{\y}(X,Y) &= \left<{\Riem}(X,E_i) E_i, Y\right> + \left<{\Riem}(X,{\sf n}) {\sf n}, Y\right>\\
		&= {^{\sf level}\Ric}(X,Y) -  \left(\nicefrac{\rho_{uu}}{\rho_u}\right)^2 \Big( (n-1)\left< X,Y \right> - \left< X, E_i \right>\left< Y, E_i  \right>   \Big)\\
		&+ \left<{\Riem}(X,{\sf n}) {\sf n}, Y\right>\\
		&= {^{\sf level}\Ric}^{\vecv}(X,Y) -\Big( (n-2) \left(\nicefrac{\rho_{uu}}{\rho_u}\right)^2   - \K^{\perp}   \Big) \left< X,Y \right>_{\!\vecv}\\
		&= {^{\sf ideal}\Ric}(X,Y) -\Big( (n-2) \mathcal{L}   - \K^{\perp}   \Big) \left< X,Y \right>_{\!\vecv},
	\end{align*}
	where we have utilized \hyperref[sec:highlights]{\secref{sec:highlights}} - \hyperref[item:curv-3]{\ref{item:curv-3}}.
	\smallskip
	\item	[]{\bf {\footnotesize \textsf{Proof of \hyperref[item:ric-curv-2]{(\ref{item:ric-curv-2})}}}.}
	\smallskip
	Straightforward from \hyperref[item:ric-curv-7]{item (\ref{item:ric-curv-7})} above. Note $\Ric({\sf n}, {\sf n})$ is $\y$-independent!
	\smallskip
	\item	[]{\bf {\footnotesize \textsf{Proof of \hyperref[item:ric-curv-3]{(\ref{item:ric-curv-3})}}}.}
	Straightforward from \hyperref[item:ric-curv-2]{item (\ref{item:ric-curv-2})} above and curvature symmetries of $\Riem$; see~\cite[Section 2.2.1]{FL2}.
	\smallskip
	\item	[]{\bf {\footnotesize \textsf{Proof of \hyperref[item:ric-curv-4]{(\ref{item:ric-curv-4})}}}.}
\par Recall that the Ricci scalar is given by
	\begin{align*}
		\mathcal{R}ic(\y) = y^j\Riem^{\;\,i}_{j\;\;il}(y)y^l.
	\end{align*}
\par By \hyperref[sec:highlights]{\secref{sec:highlights}} - \hyperref[item:highlight-6]{\ref{item:highlight-6}}, we know that $\Riem$ is independent of $y^1$, thus when $\vecv \neq 0$, we get	
\begin{align*}
	\mathcal{R}ic(\y) = y^j\Riem^{\;\,i}_{j\;\;il}(\vecv)y^l.
\end{align*}
Consequently, using the curvature relations \hyperref[sec:highlights]{\secref{sec:highlights}} - \hyperref[item:curv-1]{\ref{item:curv-1}}, when $\vecv \neq 0$, we get
\begin{align*}
		\mathcal{R}ic(\y) &=  y^j\Riem^{\;\,i}_{j\;\;il}(\y)y^l\\
		&=  y^jy^l\Riem^{\;\,\alpha}_{j\;\;\alpha l}(\y) + y^jy^l\Riem^{\;\,1}_{j\;\;1 l}(\y)\\
		&= y^\beta y^\gamma \; {\Riem}^{\;\;\alpha}_{\beta\;\; \alpha\gamma}(\vecv) + y^\beta y^\gamma\Riem^{\;\;1}_{\beta\;\;1 \gamma}(\y) + (y^1)^2\Riem^{\;\,\alpha}_{1 \;\;\alpha 1}(\y)\\
		&=  y^\beta y^\gamma\;{^{\sf ideal}\Riem}^{\;\;\alpha}_{\beta\;\; \alpha\gamma}(\vecv) \\
		& -\left(\nicefrac{\rho_{uu}}{\rho_u}\right)^2y^\beta y^\gamma\left( g_{\beta \gamma}\delta^\alpha_\alpha - g_{\beta\alpha}\delta^\alpha_\gamma \right) \\ 
		&+ \K^{\perp} g_{\beta\gamma}y^\beta y^\gamma + (n-1)\K^{\perp} (y^1)^2.
	\end{align*}
	As a result, we obtain
\begin{align}\label{eq:scalar-Ric-relation}
	\mathcal{R}ic(\y) &= y^\beta y^\gamma \; {^{\sf ideal}\Riem}^{\;\;\alpha}_{\beta\;\; \alpha\gamma}(\vecv)\\
	& - (n-2) \left(\nicefrac{\rho_{uu}}{\rho_u}\right)^2 \|\vecv\|_{\y}^2 + \K^{\perp} \|\vecv\|_{\y}^2 \notag \\
	&  - (n-1)\left(\nicefrac{\rho_{uuu}}{\rho_u}\right) (y^1)^2.\notag
\end{align} 	
\par In particular, restricted to $\vecv \in T\Sigma_0$ (i.e., setting $y^1 = 0$ and $\y = \vecv$), we get
	\begin{align*}
		\mathcal{R}ic(\vecv) &= {^{\sf ideal}\mathcal{R}ic}(\vecv) - (n-2) \left(\nicefrac{\rho_{uu}}{\rho_u}\right)^2 \|\vecv\|_{\vecv}^2 + \K^{\perp} \|\vecv\|^2_{\vecv}\\
		&= {^{\sf ideal}\mathcal{R}ic}(\vecv) - (n-2) \mathcal{L}F^2(\vecv) + \K^{\perp} F^2(\vecv).
	\end{align*}
	\item	[]{\bf {\footnotesize \textsf{Proof of \hyperref[item:ric-curv-5]{(\ref{item:ric-curv-5})}}}.}
	First we note the induced vertical double tangent vectors satisfy
	\begin{align*}
	{^\Sigma\dt{\partial}}_\alpha = \dt{\partial}_\alpha;
	\end{align*}
	see~\cite[Section 4]{FL2}.
\par From \hyperref[eq:scalar-Ric-relation]{\eqref{eq:scalar-Ric-relation}}, one gets
	\begin{align*}
		\mathcal{R}ic_{\alpha\beta}(\y) &= \nicefrac{1}{2} \, \dt{\partial}_{\beta} \dt{\partial}_{\alpha} \left(\mathcal{R}ic(\y)\right)\\
		&= \nicefrac{1}{2} \, {\dt{\partial}}_{\beta} {\dt{\partial}}_{\alpha} \bigg( y^\kappa y^\gamma\;  {^{\sf ideal}\Riem}^{\;\;\eta}_{\kappa\;\; \eta\gamma}(\vecv)  \\
		&-  (n-2) \left(\nicefrac{\rho_{uu}}{\rho_u}\right)^2 \|\vecv\|_{\y}^2 + \K^{\perp} \|\vecv\|^2_{\y}\\
		& - (n-1)\left(\nicefrac{\rho_{uuu}}{\rho_u}\right) (y^1)^2 \bigg)\\
		&= \nicefrac{1}{2} \, {^\Sigma\dt{\partial}}_{\beta} {^\Sigma\dt{\partial}}_{\alpha} \left(y^\kappa y^\gamma\;  {^{\sf ideal}\Riem}^{\;\;\eta}_{\kappa\;\; \eta\gamma}(\vecv) \right) \\
		& - \left( (n-2) \mathcal{L} - \K^{\perp} \right) g_{\alpha\beta}(\y)\\
		&= {^{\sf ideal}\mathcal{R}ic}_{\alpha\beta}(\vecv) - \left( (n-2) \mathcal{L} - \K^{\perp} \right) g_{\alpha\beta}(\y).
	\end{align*}
\par By the fact that $g_{\alpha\beta}$ is independent of $y^1$, we infer that for $\vecv \neq 0$, we have
\begin{align*}
g_{\alpha\beta}(\y) = g_{\alpha\beta}(\vecv).
\end{align*} 
\par Similarly,
	\begin{align*}
		\mathcal{R}ic_{11}(\y) &= \nicefrac{1}{2} \, \dt{\partial}_{u} \dt{\partial}_{u} \left( {^{\sf ideal}\mathcal{R}ic}(\vecv) -  (n-2) \left(\nicefrac{\rho_{uu}}{\rho_u}\right)^2 \|\vecv\|_{\vecv}^2 - (n-1)\left(\nicefrac{\rho_{uuu}}{\rho_u}\right) (y^1)^2\right)\\
		&=  -(n-1)\left(\nicefrac{\rho_{uuu}}{\rho_u}\right) = (n-1)\K^{\perp};
	\end{align*}
and
	\begin{align*}
	\mathcal{R}ic_{\alpha1}(\y) &= \nicefrac{1}{2} \, \dt{\partial}_{\alpha} \dt{\partial}_{u} \Big( {^{\sf ideal}\mathcal{R}ic}(\vecv) - (n-2) \left(\nicefrac{\rho_{uu}}{\rho_u}\right)^2 \|\vecv\|_{\y}^2 - (n-1)\left(\nicefrac{\rho_{uuu}}{\rho_u}\right) (y^1)^2 \Big)\\
	&= \dt{\partial}_{\alpha} \Big( (n-1)\left(\nicefrac{\rho_{uuu}}{\rho_u}\right) y^1\Big)\\
	&=0.
\end{align*}
Thus, $\Ric$ and $\mathcal{R}ic_{\sf tensor}$ satisfy the same tangential and normal curvature relations.
\smallskip 
\item []{\bf {\footnotesize \textsf{Proof of \hyperref[item:ric-curv-6]{(\ref{item:ric-curv-6})}}}.}
By taking the trace and using items \hyperref[item:ric-curv-1]{(\ref{item:ric-curv-1})} and \hyperref[item:ric-curv-2]{(\ref{item:ric-curv-2})} above. 
\smallskip
\item []{\bf {\footnotesize \textsf{Proof of \hyperref[item:ric-curv-7]{(\ref{item:ric-curv-7})}}}.}
By taking the trace and using \hyperref[item:curv-5]{item (\ref{item:ric-curv-5})} above.
\end{enumerate}
\end{proof}	

\begin{lemma}\label{lem:const-scalar-one-crit}
	Suppose $\sigma$ has at least one critical point and $(M,F)$ has constant (Cartan or Akbarzadeh) scalar curvature. Then, both Cartan and Akbarzadeh scalar curvature are constant and coincide. Furthermore, the warping function $\zeta := \rho_u$ is given by
	\begin{align*}
	 \zeta = 
	\begin{cases}
		\left(\nicefrac{{\mathdutchcal{s}}_{\sf level}}{{\mathdutchcal{s}}}\right)^{\nicefrac{1}{2}}	\sin\left( {\mathdutchcal{s}}^{\nicefrac{1}{2}} \, u \right) & {\mathdutchcal{s}} > 0\\
		{{\mathdutchcal{s}}^{\nicefrac{1}{2}}_{{\emph{\textsf{level}}}}}\, u &  {\mathdutchcal{s}} = 0\\
		\left(\nicefrac{{\mathdutchcal{s}}_{\sf level}}{\left|{\mathdutchcal{s}}\right|}\right)^{\nicefrac{1}{2}}	\sinh\left( \left|\mathdutchcal{s}\right|^{\nicefrac{1}{2}} \, u \right) & {\mathdutchcal{s}} < 0
	\end{cases},
	\end{align*}
where ${\mathdutchcal{s}}_{\sf level}$ denotes the constant normalized scalar curvature of the ideal hypersurface. 
 In particular, 
\begin{align*}
	\mathsf{s} = \mathdutchcal{s} = \K^\perp.
\end{align*}
\end{lemma}
\begin{proof}[\small \textbf{Proof}]
	By \hyperref[sec:highlights]{\secref{sec:highlights}} - \hyperref[item:const-curv-4]{\ref{item:const-curv-4}}, we know that $\left( \Sigma, g_{\sf ideal} \right)$ has constant flag curvature $\rho^2_{uu}(0)$; thus 
	\begin{align*}
	{\mathdutchcal{s}_{\sf ideal}} = \K_{\sf ideal} = \rho^2_{uu}(0).
	\end{align*}
	\par The equation \hyperref[eq:scalar-normalized]{\eqref{eq:scalar-normalized}}, reduces to
	\begin{align*}
		\mathdutchcal{s}\rho_u^2 = - \nicefrac{2}{n}\; \rho_u\rho_{uuu} + \nicefrac{(n-2)}{n} \; \mathdutchcal{s}_{\sf ideal} \, - \, \nicefrac{(n-2)}{n} \;\rho_{uu}^2.
	\end{align*}
	By the change of variable $\zeta := \rho_u$, we get
	\begin{align*}
	\zeta\zeta_{uu} + \mathdutchcal{s}\zeta^2  - \nicefrac{(n-2)}{2} \; {\mathdutchcal{s}_{\sf ideal}} + \, \nicefrac{(n-2)}{2} \;\zeta_{u}^2 =  0.
	\end{align*}
\par Now, similar to the proof of \cite[Theorem 24]{Kuh}, multiplying by the integrating factor $2\zeta^{n-3}\zeta_u$,  the equation reduces to
	\begin{align*}
	\zeta^{n-2}\left( \zeta_u^2 + {\mathdutchcal{s}} \zeta^2 -  {\mathdutchcal{s}_{\sf ideal}}   \right) = 0,
	\end{align*}
	thus on the regular domain, one obtains
	\begin{align*}
	\zeta_u^2 + {\mathdutchcal{s}} \zeta^2 -  {\mathdutchcal{s}_{\sf ideal}} = 0,
	\end{align*}
	which upon differentiation gives
	\begin{align*}
	2\zeta_u\zeta_{uu} + 2{\mathdutchcal{s}} \zeta\zeta_u = 0.
	\end{align*}
	\par Therefore, we have obtained the Fourier equation
	\begin{align*}
	\begin{cases}
		\zeta_{uu} + {\mathdutchcal{s}} \zeta = 0\\
		\zeta(0) = 0\\
		\zeta_{u}(0) = \sqrt{{\mathdutchcal{s}}_{\sf ideal}}
	\end{cases}.
	\end{align*}
	The possible solutions thus are in the form of $c \; \mathrm{sn}_{\kappa}$ function for $\kappa = \mathdutchcal{s}$, more precisely,
	\begin{align*}
	\rho_u = \zeta = 
	\begin{cases}
		\left(\nicefrac{{\mathdutchcal{s}}_{\sf ideal}}{{\mathdutchcal{s}}}\right)^{\nicefrac{1}{2}}	\sin\left( {\mathdutchcal{s}}^{\nicefrac{1}{2}} \, u \right) & {\mathdutchcal{s}} > 0\\
		{{\mathdutchcal{s}}^{\nicefrac{1}{2}}_{\textsf{ideal}}}\, u &  {\mathdutchcal{s}} = 0\\
		\left(\nicefrac{{\mathdutchcal{s}}_{\sf ideal}}{\left|{\mathdutchcal{s}}\right|}\right)^{\nicefrac{1}{2}}	\sinh\left( \left|\mathdutchcal{s}\right|^{\nicefrac{1}{2}} \, u \right) & {\mathdutchcal{s}} < 0
	\end{cases}.
	\end{align*}
\par The last claim follows by noting that
	\[
	\K^\perp = \nicefrac{\zeta_{uu}}{\zeta} = \mathdutchcal{s}
	\]
	and the similar relation in terms of $\mathsf{s}$.
\end{proof}
\begin{lemma}\label{lem:vanishing-cross-scalar}
	Suppose $\sigma$ has no critical points, $\K^{\perp}$ (recall the normal flag curvature) is constant, and $(M,F)$ has constant normalized Akbarzadeh scalar curvature $\mathdutchcal{s} = \K^{\perp}$ (resp. constant normalized Cartan scalar curvature ${\sf s} = \K^{\perp}$ ). Then, $\left( \Sigma, g_{\sf ideal} \right)$ is Ricci scalar flat (resp. Cartan Ricci flat). 
\end{lemma}
\begin{proof}[\small \textbf{Proof}]
	Upon integrating
	\begin{align*}
		\rho_{uuu} = -\K^{\perp} \rho_u,
	\end{align*}
	we get
	\begin{align*}
		\rho_{uu} = -\K^{\perp} \rho + C.
	\end{align*}
	By the fact that $\rho$ and $\rho_u$ are both nowhere vanishing on $\R$, we deduce that $\K^{\perp} < 0$. The characteristic roots are $\lambda = \pm |\K^{\perp}|^{\nicefrac{1}{2}}$; thus, the complementary solution is
	\begin{align*}
		\rho_{h} = c_1e^{-u|\K^{\perp}|^{\nicefrac{1}{2}}} + c_2e^{u|\K^{\perp}|^{\nicefrac{1}{2}}} = c_1\sinh(u|\K^{\perp}|^{\nicefrac{1}{2}}) + c_2\cosh(u|\K^{\perp}|^{\nicefrac{1}{2}}),
	\end{align*}
 for which, $\rho_{p} = \nicefrac{C}{\K^{\perp}}$ is a particular solution.
Therefore,
\begin{align*}
\rho = c_1e^{-u|\K^{\perp}|^{\nicefrac{1}{2}}} + c_2e^{u|\K^{\perp}|^{\nicefrac{1}{2}}} + \nicefrac{C}{\K^{\perp}}
\end{align*}
\par Since by definition $\rho >0$ and considering the asymptotic as $u \to \pm \infty$, one deduces that $c_1,c_2 \ge 0$. Noting $\rho_u<0$, we deduce that $c_2 = 0$ therefore,
\begin{align*}
	\rho = c_1e^{-u{\left|\K^{\perp}\right|}^{\nicefrac{1}{2}}} + \nicefrac{C}{\K^{\perp}},
\end{align*}
and $\nicefrac{C}{\K^{\perp}} \ge 0$; otherwise, $\rho$ would admit a zero. 
\par We have concluded that
\begin{align*}
\zeta = \rho_u = -c_1|\K^{\perp}|^{\nicefrac{1}{2}}e^{-u|\K^{\perp}|^{\nicefrac{1}{2}}}, \quad c_1\ge 0
\end{align*}
therefore,
\begin{align*}
g = du^2 + c_1^2\left| \K^{\perp} \right| e^{-2u|\K^{\perp}|^{\nicefrac{1}{2}}} g_{\sf ideal}.
\end{align*}
Form \hyperref[eq:scalar-normalized]{\eqref{eq:scalar-normalized}}, and by the hypothesis $\mathdutchcal{s} = \K^{\perp}$ we obtain
\begin{align*}
{^{\sf level}\mathdutchcal{s}} =	\mathdutchcal{s} + \mathcal{L} =  \mathdutchcal{s} + \left(\nicefrac{\zeta_{u}}{\zeta}\right)^2 = \K^{\perp} + \left| \K^{\perp} \right| = 0.
\end{align*}
The statement in terms of the Cartan scalar curvature ${\sf s}$ follows by a verbatim argument. 
\end{proof}
\subsection*{Proof of the \hyperref[thrm:classification]{Main Theorem}}
\hfill
\par The fact that $\sigma$ has at most two critical points and that regular hypersurfaces are diffeomoprhic, are established in \hyperref[lemma:crit]{Lemma~\ref{lemma:crit}}. 
\par 	As an immediate consequence of \hyperref[lem:curvature]{Lemma~\ref{lem:curvature}} and items \hyperref[item:curv-5]{\ref{item:curv-5}} - \hyperref[item:curv-6]{\ref{item:curv-6}} from \hyperref[sec:highlights]{\secref{sec:highlights}}, one obtains the following. 
	\begin{itemize}
	\item \label{item:const-curv-inherit-1} If $(M,F)$ has isotropic (resp. constant)  flag curvature, then $\left(\Sigma, {g}_{\sf ideal} \right)$ also has isotropic flag curvature (resp. constant flag curvature and $\mathdutchcal{s} = \K^{\perp}$), \underline{but, not vice-versa}.
	\medskip
	\item \label{item:const-curv-inherit-2} If $(M,F)$ has sectional flag curvature, then $\left(\Sigma, {g}_{\sf ideal} \right)$ also has sectional flag curvature, \underline{but, not vice-versa}.
	\medskip
	\item \label{item:const-curv-inherit-3} If $(M,F)$ is Einstein (resp. is Cartan-Einstein), then $\left(\Sigma, g_{\sf ideal}\right)$ is Einstein (resp. is Cartan-Einstein) and $\mathdutchcal{s} = K(x) =  \K^{\perp}$ (resp. $\mathsf{s} = K(x) = \K^{\perp}$); \underline{but, not vice-versa}.
	\medskip
	\item \label{item:const-curv-inherit-4} If $(M,F)$ is strongly Einstein (resp. is strongly Cartan-Einstein), then $\left(\Sigma, g_{\sf ideal}\right)$ is strongly Einstein (resp. is strongly Cartan-Einstein) and $\mathdutchcal{s} = K = \K^{\perp}$ is constant (resp. $\mathsf{s} = K = \K^{\perp}$ is constant); \underline{but, not vice-versa}. 
	\medskip
	\item \label{item:const-curv-inherit-5} If $(M,F)$ has constant Cartan scalar curvature (resp. Akbarzadeh scalar curvature) then, $\left(\Sigma, g_{\sf ideal}\right)$ also has constant Cartan scalar curvature (resp. Akbarzadeh scalar curvature), \underline{but, not vice-versa}.
	\medskip
	\item \label{item:const-curv-inherit-6} \emph{$(M,F)$ has vector-independent Cartan scalar curvature (resp. Akbarzadeh scalar curvature), if and only if, $\left(\Sigma, g_{\sf ideal}\right)$ also has vector-independent Cartan scalar curvature (resp. Akbarzadeh scalar curvature).}
	\end{itemize} 
\subsubsection*{{\bf{\footnotesize \textsf{Proof of \hyperref[item:class-1]{\textrm{Case~$\sf I$}}}}}}
\begin{enumerate}
	\item []{\footnotesize \textsf{Proof of  \hyperref[item:class-1-a]{\textrm{Case~$\sf I$a}}}}. This is the content of \hyperref[lemma:crit]{Lemma~\ref{lemma:crit}} - \hyperref[item:level-diff-2]{item (\ref{item:level-diff-2})}.
	\smallskip 
	\item []{\footnotesize \textsf{Proof of \hyperref[item:class-1-b]{\textrm{Case~$\sf I$b}}}}.
	Since there exist no critical points, the normal vector field ${\sf n}$ is a global vector field. By the characterization Theorem~\ref{thrm:Riem-linearization}, $g^{\sf n}$ is then a global Riemannian \cpt. 
	\smallskip
	\item []{\footnotesize \textsf{Proof of \hyperref[item:class-1-c]{\textrm{Case~$\sf I$c}}}}.
	Based on \hyperref[item:class-1-b]{\textrm{\footnotesize Case~$\sf I$b}} above, this follows from the Riemannian classification (see \hyperref[sec:intro]{\secref{sec:intro}}, page 1) once we argue that  $(M, g^n)$ is complete. 
 \par By \underline{tameness} we deduce that there exists a $0 < \Lambda_1 \le \Lambda_2 < \infty$ such that 
\begin{align*}
	\Lambda_1 < g_{\y}({\footnotesize \textit{\textbf{u}}},{\footnotesize \textit{\textbf{u}}}) < \Lambda_2,
\end{align*}
whenever $F(\y) = F(\footnotesize \textit{\textbf{u}}) = 1$.
This implies
\begin{align*}
F^2(u) \Lambda_1 < g^{\sf n}({\footnotesize \textit{\textbf{u}}},{\footnotesize \textit{\textbf{u}}}) < \Lambda_2 F^2(u).
\end{align*}
\par Consequently, the forward or backward length of a smooth curve measured w.r.t. $g^{\sf n}$ and w.r.t. $g$ are comparable. Upon using the definition of distance by minimizing curve lengths, one deduces that $\dist_{g^{\sf n}}$ and $\dist_F$ (forward or backward) are equivalent distances. Since $(M,F)$ is (forward or backward) complete, we deduce that $(M,g^{\sf n})$ is a complete Riemannian manifold.
\item []{\footnotesize \textsf{Proof of \hyperref[item:class-1-d]{\textrm{Case~$\sf I$d}}}}. This is established in \hyperref[lem:vanishing-cross-scalar]{Lemma~\ref{lem:vanishing-cross-scalar}}.
\end{enumerate}
\subsubsection*{{\bf{\footnotesize \textsf{Proof of \hyperref[item:class-2]{\textrm{Case~$\sf II$}}}}}}
\begin{enumerate}
	\item []{\footnotesize \textsf{Proof of \hyperref[item:class-2-a]{\textrm{Case~$\sf II$a}}}}. 
	Suppose the only critical point is an absolute minimum (see \hyperref[lemma:crit]{Lemma~\ref{lemma:crit}} - \hyperref[item:level-diff-1]{item (\ref{item:level-diff-1})}), then by \hyperref[lemma:crit]{Lemma~\ref{lemma:crit}} - \hyperref[item:level-diff-2]{item (\ref{item:level-diff-2})} and \hyperref[sec:highlights]{\secref{sec:highlights}} - \hyperref[item:const-curv-4]{\ref{item:const-curv-4}}, we know that $M\smallsetminus \{p\}$ is diffeomorphic to $\R^{\ge 0} \times \Sp^{n-1}$. Thus adding $p$ back, we get $M \cong \R^n$. When $p$ is an absolute maximum, the proof is similar.
	\smallskip 	
\item []{\footnotesize \textsf{Proof of \hyperref[item:class-2-b]{\textrm{Case~$\sf II$b}}}}.
This was established in \hyperref[lemma:crit]{Lemma~\ref{lemma:crit}}
\smallskip
\item []{\footnotesize \textsf{Proof of \hyperref[item:class-2-c]{\textrm{Case~$\sf II$c}}}}.
By the hypothesis and using \hyperref[lem:const-scalar-one-crit]{Lemma~\ref{lem:const-scalar-one-crit}}, we infer that the warping function $\rho_u = \zeta$ is of the form,
	\begin{align*}
		\rho_u = \zeta = 
		\begin{cases}
			{{\mathdutchcal{s}}^{\nicefrac{1}{2}}_{\textsf{level}}}\, u &  {\mathdutchcal{s}} = 0\\
			\left(\nicefrac{{\mathdutchcal{s}}_{\sf level}}{\left|{\mathdutchcal{s}}\right|}\right)^{\nicefrac{1}{2}}	\sinh\left( \left|\mathdutchcal{s}\right|^{\nicefrac{1}{2}} \, u \right) & {\mathdutchcal{s}} < 0
		\end{cases},
	\end{align*}
	since the case $ {\mathdutchcal{s}} > 0$ would have implied more than one critical point. 
\par In any case, ${\mathdutchcal{s}} = - \nicefrac{\rho_{uuu}}{\rho_u} = \K^{\perp}$ holds true. By \hyperref[item:class-2-b]{\footnotesize \textrm{Case~$\sf II$b}} above, \hyperref[sec:highlights]{\secref{sec:highlights}} - \hyperref[item:curv-5]{\ref{item:curv-5}} and \hyperref[eq:scalar-normalized]{\eqref{eq:scalar-normalized}}, one deduces that
\begin{align*}
	\K(X\wedge Y) &=  {^{\textsf{level}}\K} - \left(\nicefrac{\rho_{uu}}{\rho_u}\right)^2 \\
	&= 	{^{\sf level}{\mathdutchcal{s}}} - \left(\nicefrac{\rho_{uu}}{\rho_u}\right)^2 \\
	&= \K^{\perp},
\end{align*} 
and the conclusion follows by noting that both Cartan and Chern curvature tensors give rise to the same flag curvatures. 
\end{enumerate}
\subsubsection*{{\bf{\footnotesize \textsf{Proof of \hyperref[item:class-3]{\textrm{Case~$\sf III$}}}}}}
\begin{enumerate}
	\item []{\footnotesize \textsf{Proof of \hyperref[item:class-3-a]{\textrm{Case~$\sf III$a}}}}.
	Let $p,q$ be the critical points; $p$ is the absolute minimum and $q$ is the absolute maximum. Then by \hyperref[lemma:crit]{Lemma~\ref{lemma:crit}}, $M \smallsetminus \left\{p,q\right\}$ is diffeomorphic to $(a,b)\times \Sp^{n-1}$. Thus adding back the points $p$ and $q$, we deduce that $M$ is diffeomorphic to $\Sp^n$.
	\smallskip  
	\item []{\footnotesize \textsf{Proof of \hyperref[item:class-3-b]{\textrm{Case~$\sf III$b}}}}.
	This is established in \cite[Lemma~5.1]{FL2}; also see \hyperref[sec:highlights]{\secref{sec:highlights}} - \hyperref[item:const-curv-4]{\ref{item:const-curv-4}}
	\smallskip
	\item []{\footnotesize \textsf{Proof of \hyperref[item:class-3-c]{\textrm{Case~$\sf III$c}}}}.
	Since $\sigma$ has two critical points, using \hyperref[lem:const-scalar-one-crit]{Lemma~\ref{lem:const-scalar-one-crit}}, we must have ${\mathdutchcal{s}} > 0$ and 
	\begin{align*}
		\zeta = 
		\left(\nicefrac{{\mathdutchcal{s}}_{\sf level}}{{\mathdutchcal{s}}}\right)^{\nicefrac{1}{2}} \sin\left( {\mathdutchcal{s}}^{\nicefrac{1}{2}} \, u \right).
	\end{align*}
	Consequently, $\K^{\perp} = - \nicefrac{\zeta_{uu}}{\zeta} = {\mathdutchcal{s}}$. In the same fashion as in the proof of \hyperref[item:class-2-c]{\footnotesize\textrm{Case~$\sf II$c}} above, one obtains
	\begin{align*}
		\K(X\wedge Y) = \K^{\perp} = \mathdutchcal{s},
	\end{align*} 
	and the conclusion follows.
\end{enumerate}
\phantom{}\hfill \scalebox{1.2}{\qed} 

\vspace{10mm}
\end{document}